\documentclass[a4paper]{amsart}
\usepackage{amssymb, stmaryrd, hyperref}

\newcommand{\cO}{\mathcal{O}}

\newcommand{\QQ}{\mathbb{Q}}
\newcommand{\RR}{\mathbb{R}}
\newcommand{\ZZ}{\mathbb{Z}}
\newcommand{\rig}{\mathrm{rig}}

\DeclareMathOperator{\Rig}{Rig}
\DeclareMathOperator{\GL}{GL}
\DeclareMathOperator{\Spf}{Spf}
\DeclareMathOperator{\Spec}{Spec}

\newcommand{\into}{\hookrightarrow}
\newcommand{\onto}{\twoheadrightarrow}

\newtheorem{proposition}{Proposition}[section]
\newtheorem{theorem}[proposition]{Theorem}

\newtheorem*{assumption}{Assumption}

\theoremstyle{definition}
\newtheorem{definition}[proposition]{Definition}

\theoremstyle{remark}
\newtheorem{remark}[proposition]{Remark}

\begin{document}

\title{Density of classical points in eigenvarieties}
\author{David Loeffler}
\address{Warwick Mathematics Institute \\ Zeeman Building \\ University of Warwick \\ Coventry CV4 7AL, UK}
\email{d.a.loeffler@warwick.ac.uk}
\thanks{Supported by EPSRC postdoctoral fellowship EP/F04304X/2}

\begin{abstract}
In this short note, we study the geometry of the eigenvariety parametrising $p$-adic automorphic forms for $\GL_1$ over a number field, as constructed by Buzzard. We show that if $K$ is not totally real and contains no CM subfield, points in this space arising from classical automorphic forms (i.e.~algebraic Gr\"ossencharacters of $K$) are not Zariski-dense in the eigenvariety (as a rigid space); but the eigenvariety posesses a natural formal scheme model, and the set of classical points \emph{is} Zariski-dense in the formal scheme.

We also sketch the theory for $\GL_2$ over an imaginary quadratic field, following Calegari and Mazur, emphasising the strong formal similarity with the case of $\GL_1$ over a general number field.
\end{abstract}

\maketitle

\section{Zariski-density in formal and rigid spaces}

Let $A$ be a finite algebra over the formal power series ring $\ZZ_p \llbracket T_1, \dots, T_n \rrbracket$ for some $n \ge 0$, which is flat over $\ZZ_p$ (i.e.~$p$ is not a zero-divisor in $A$). Then we have a choice of geometric objects attached to $A$: the affine scheme $\Spec(A)$, and its generic fibre $\Spec(A[\frac{1}{p}])$; the affine formal scheme $\Spf(A)$; and the rigid-analytic space $(\Spf(A))^\rig$ obtained by applying Berthelot's generic fibre construction \cite[\S 7]{dejong-formalrigid}. We abbreviate the latter by $\Rig(A)$.

\begin{proposition}
 There following three sets are in canonical bijection with each other:
 \begin{itemize}
  \item Points of $\Spec(A[\frac{1}{p}])$, i.e.~maximal ideals of $A[\frac{1}{p}]$;
  \item Morphisms of formal schemes $\Spf(\cO) \to \Spf(A)$, with $\cO$ the ring of integers of a finite extension of $\QQ_p$ (``rig-points'' of $\Spf(A)$);
  \item Points of $\Rig(A)$.
 \end{itemize}
\end{proposition}

\begin{proof} See \cite[7.1.9,7.1.10]{dejong-formalrigid}.
\end{proof}

We do not, however, obtain bijections between closed subvarieties of these geometric objects; closed subschemes of $\Spec(A[\frac{1}{p}])$ biject with closed formal subschemes of $\Spf A$ flat over $\ZZ_p$, but these correspond to a subset of the closed subvarieties of $\Rig(A)$. Most of the content of the present note relates in some way or another to the following key example. If $A = \ZZ_p \llbracket T \rrbracket$, then $\Rig(A)$ is the rigid-analytic open unit disc, and the set of points $T$ such that $(1 + T)^{p^n} = 1$ for some $n \in \mathbb{N}$ is a closed subvariety of $\Rig(A)$ (cut out by the $p$-adic logarithm $\log(1 + T)$), which is clearly not the analytification of any closed subvariety of $\Spf(A)$.

If $P(A)$ is the common set of points from the preceding proposition, we refer to the topology on $P(A)$ whose closed subsets are given by ideals of $A[\frac{1}{p}]$ (or, equivalently, closed subvarieties of $\Spf(A)$ flat over $\ZZ_p$) as the \emph{formal Zariski topology}, and the topology whose closed subsets are given by rigid-analytic subvarieties of $\Rig(A)$ as the \emph{rigid Zariski topology}. As the preceding example shows, the rigid Zariski topology may be strictly finer than the formal Zariski topology.

\section{Character spaces}

Let $G$ be an abelian $p$-adic analytic group; equivalently, $G$ is any group of the form $\ZZ_p^d \times H$, for $d \ge 0 \in \ZZ$ and $H$ a finite abelian group.

\begin{theorem}
 The functor mapping an Artinian local $\ZZ_p$-algebra $A$ to the set of continuous group homomorphisms $G \to A^\times$ is pro-representable, and is represented by the formal scheme $\widehat G = \Spf \ZZ_p\llbracket G \rrbracket$, where $\ZZ_p\llbracket G \rrbracket$ is the Iwasawa algebra of $G$, equipped with the canonical character $G \to \ZZ_p\llbracket G\rrbracket^\times$. Moreover, the generic fibre $\widehat G^\rig$ of $\widehat G$ is the rigid space constructed in \cite[Lemma 2]{buzzard-families} which represents the corresponding functor on the category of affinoid $\QQ_p$-algebras.
\end{theorem}

\begin{proof}
Essentially by definition, any continuous homomorphism $G \to A^\times$ extends uniquely to a ring homomorphism $\ZZ_p\llbracket G\rrbracket \to A$, and conversely any ring homomorphism $\ZZ_p\llbracket G\rrbracket \to A$ gives a group homomorphism $G \to A^\times$ by composition with the canonical character (which is continuous, since $A$ is Artinian). Furthermore, $\ZZ_p\llbracket G\rrbracket$ can clearly be written as an inverse limit of the quotients $(\ZZ/p^n \ZZ)\llbracket G / U\rrbracket$ for $U$ open in $G$, which are Artinian $\ZZ_p$-algebras. Moreover, if $G_1$ and $G_2$ are two such groups, we have 
\[\ZZ_p\llbracket G_1 \times G_2\rrbracket = \ZZ_p\llbracket G_1\rrbracket \hat\otimes_{\ZZ_p} \ZZ_p\llbracket G_2\rrbracket;\]
the generic fibre construction commutes with fibre products, so it suffices to check that the generic fibre of $\Spf \ZZ_p\llbracket G\rrbracket$ agrees with Buzzard's construction when $G$ is either $\ZZ_p$ or a finite cyclic group; both of these cases are easy.
\end{proof}

Now let $K$ be a number field. We define
\[ \cO_{K, p}^\times := (\cO_K \otimes \ZZ_p)^\times = \prod_{v \mid p} \cO_{K, v}^\times.\]
It is clear that $\cO_{K, p}^\times$ is an abelian $p$-adic analytic group of dimension $d = [K : \QQ]$; we let $\mathcal{W} = \widehat{\cO_{K, p}^\times}$. A point of $\mathcal{W}$ is thus equivalent to a continuous homomorphism $\cO_{K, p}^\times \to E^\times$, for $E$ some finite extension of $\QQ_p$; we refer to these as \emph{$p$-adic weights} for $K$.

Let $K_\infty^\circ$ be the identity component of $(K \otimes \RR)^\times$, and $U$ any open compact subgroup of $\left(\mathbb{A}_K^{p, \infty}\right)^\times$. We define
\[ H(U) = \mathbb{A}_K^\times / \overline{K^\times \cdot U \cdot K^\circ_\infty}.\]

\begin{definition}\cite{buzzard-families}
 The \emph{eigenvariety} for $\GL_1 / K$ of tame level $U$ is the formal $\ZZ_p$-scheme $\mathcal{E}(U) = \widehat{H(U)}$.
\end{definition}

The inclusion $\cO_{K, p}^\times \into \mathbb{A}_K^\times$ gives a continous map $\cO_{K, p}^\times \to H(U)$ whose kernel is the closure in $\cO_{K, p}^\times$ of the abelian group $\Gamma(U) = K^\times \cap \left(U \cdot \cO_{K, p}^\times \cdot K_\infty^\circ\right)$. The cokernel of this map is finite (it is the ray class group modulo $UK_\infty^\circ$) and hence $H(U)$ is also a compact abelian $p$-adic analytic group, of dimension equal to $1 + r_2 + \delta$ where $r_2$ is the number of complex places of $K$ and $\delta$ is the defect in Leopoldt's conjecture for $K$ at $p$.

If we write $Q(U) = \cO_{K, p}^\times / \overline{\Gamma(U)}$, then we can identify $Q(U)$ with a finite-index subgroup of $H(U)$; hence we have maps $\ZZ_p\llbracket \cO_{K, p}^\times \rrbracket \onto \ZZ_p\llbracket Q(U)\rrbracket \into \ZZ_p\llbracket H(U)\rrbracket$, where the second map is finite and flat (and becomes \'etale after inverting $p$). Thus the morphism $\mathcal{E}(U) \to \mathcal{W}$ factors as a finite flat surjective map followed by the inclusion of the closed subscheme $\mathcal{W}(U) = \widehat{Q(U)}$ of $\mathcal{W}$. In particular, we have the following result:

\begin{proposition}\label{dimension}
 Every component of $\mathcal{E}(U)$ has dimension equal to $1 + r_2 + \delta$, and maps surjectively to a component of $\mathcal{W}(U)$.
\end{proposition}

(Note that $\mathcal{E}(U)$ is not flat over $\mathcal{W}$ unless $r_1 + r_2 = 1$, i.e. $K$ is either $\QQ$ or an imaginary quadratic field.)

\section{Algebraic points}

Let $\kappa$ be a $p$-adic weight for $K$. We say that $\kappa$ is \emph{algebraic} if we can write
\[ \kappa(x) = \prod_{i} \sigma_i(x)^{n_i}\]
where $\sigma_1, \dots, \sigma_d$ are the ring homomorphisms $\cO_{K, p} \to \overline{\QQ}_p$ arising from the $d$ embeddings $K \into \overline{\QQ}_p$, and $n_i \in \ZZ$. We say $\kappa$ is \emph{parallel} if $\kappa$ factors through the norm map $N_{K/\QQ}$ (extended $\ZZ_p$-linearly to a ring homomorphism $\cO_{K, p} \to \ZZ_p$). Note that an algebraic weight is parallel if and only if the $n_i$ are all equal.

If $\kappa$ is algebraic in the above sense when restricted to some open neighbourhood of the identity, we say $\kappa$ is \emph{locally algebraic}; this is equivalent to the existence of a factorisation $\kappa = \varepsilon \kappa'$ where $\kappa'$ is algebraic and $\varepsilon$ has finite order. Similarly, if $\kappa$ is a $p$-adic weight which becomes parallel when restricted to some open neighbourhood of the identity, we say $\kappa$ is \emph{locally parallel}. 

Let us fix an isomorphism between $\mathbb{C}$ and $\overline{\QQ}_p$. Then there is a bijection between algebraic Gr\"ossencharacters of $K$ (of level containing $U$) and points of $\mathcal{E}(U)$ whose projection to $\mathcal{W}$ is locally algebraic. This maps a Gr\"ossencharacter of infinity-type $x \mapsto \prod_i \sigma_i(x)^{n_i}$ to a locally algebraic character with the same algebraic part.

We make the following assumption, which will remain in force for the remainder of this section:

\begin{assumption}
 The field $K$ contains no CM subfield.
\end{assumption}

\begin{theorem}[{\cite{weil56}}] \label{weilthm}
If the above assumption holds, then the infinity-type of every algebraic Gr\"ossencharacter of $K$ is parallel (i.e.~factors through the norm map $\operatorname{N}_{K/\QQ}$).
\end{theorem}

If $\kappa$ is a locally algebraic weight, we define $c(\kappa)$ to be the smallest integer $r \ge 0$ such that $\kappa$ is algebraic when restricted to $1 + p^r \cO_{K, p}$.

\begin{proposition}
 For any $N < \infty$, there is a 1-dimensional closed formal subscheme of $\mathcal{W}$ that contains every locally algebraic weight $\kappa \in \mathcal{W}(U)$ with $c(\kappa) \le N$.
\end{proposition}

\begin{proof}
If $\kappa \in \mathcal{W}(U)$ is locally algebraic, then by Weil's theorem it must be of the form $x \mapsto \varepsilon(x) N_{K / \QQ}(x)^k$ for some $k \in \mathbb{Z}$ and finite-order $\varepsilon$. Since the subgroup $\overline{\Gamma(U)} \cdot (1 + p^N \cO_{K, p})$ has finite index in $\cO_{K, p}^\times$, there are only finitely many candidates for $\varepsilon$. Hence the locally algebraic weights with $c(\kappa) \le N$ are contained in the union of finitely many translates of the 1-dimensional subscheme $\mathcal{W}_0 \subseteq \mathcal{W}$ parametrising parallel weights (which is simply the space of characters of $N_{K / \QQ} \left(\cO_{K, p}^\times\right) \subseteq \ZZ_p^\times$).
\end{proof}

We assume henceforth that $K$ is not totally real, so $\mathcal{W}(U)$ has dimension $1 + r_2 + \delta > 1$. It follows that the locally algebraic weights with $c(\kappa) \le N$ are not dense in the formal Zariski topology of $\mathcal{W}(U)$. In particular, for a fixed coefficient field $E$ which is discretely valued, the set of $E$-valued finite-order characters is finite (since $E$ contains finitely many $p$-power roots of unity) and thus the locally algebraic $E$-valued weights are not formally Zariski-dense in $\mathcal{W}(U)$. 

\begin{proposition}
 The closure of the locally algebraic weights in the rigid Zariski topology of $\mathcal{W}(U)$ is a closed rigid subvariety of $W(U)$ of dimension 1. However, this set is dense in the formal Zariski topology of $\mathcal{W}(U)$.
\end{proposition}

\begin{proof}
 Let $u_1, \dots, u_{d-1}$ be a $\ZZ_p$-basis for the torsion-free part of the subgroup
\[ C = \left\{ x \in \cO_{K, p}^\times : N_{K/\QQ}(x) = 1\right\}.\]

The functions $\kappa \mapsto \log(\kappa(u_i))$ are analytic functions on $\mathcal{W}^\rig$. Moreover, the derivatives of these functions are linearly independent at the origin, and hence anywhere (since they are homomorphisms of rigid-analytic group varieties). Thus they cut out a reduced rigid subvariety of $\mathcal{W}^\rig$ of dimension 1. I claim that every locally algebraic point of $\mathcal{W}(U)$ lies in this subvariety. Indeed, suppose $\kappa$ is such a point, with residue field $E$. Then $\kappa(C)$ must be finite, since the algebraic part of $\kappa$ is trivial on $C$. Therefore $\kappa(u_1), \dots, \kappa(u_d)$ must be roots of unity in $E^\times$; as the subgroup of $C$ generated by the $u_i$ is pro-$p$, these must be $p$-power roots of unity. Hence they are zeros of the $p$-adic logarithm.

On the other hand, the even powers of the norm character $\cO_{K, p}^\times \to \ZZ_p^\times$ are clearly in $\mathcal{W}(U)$, and the closure of these (in either the formal or the rigid Zariski topology) is a formal subscheme of $\mathcal{W}$ of dimension 1; so the dimension of the rigid Zariski closure of the locally algebraic weights in $\mathcal{W}(U)$ is exactly 1.

For the second statement, since $\mathcal{W}(U) = \widehat{Q(U)}$ is affine, it suffices to check that there is no nonzero element of $\ZZ_p\llbracket Q(U)\rrbracket$ whose image under any locally constant character is zero. This is clear since $\ZZ_p\llbracket Q(U)\rrbracket$ is by construction the inverse limit of the $\ZZ_p$-group rings of the finite quotients of $Q(U)$.
\end{proof}

We now lift these statements to $\mathcal{E}(U)$. If $\chi$ is a point of $\mathcal{E}(U)$, we say $\chi$ is locally algebraic if its image $\kappa \in \mathcal{W}(U)$ is so (equivalently, if it corresponds to an algebraic Gr\"ossencharacter of $K$); if this is the case, we define $c(\chi) = c(\kappa)$, which is the smallest power of $p$ divisible by the $p$-part of the conductor of the corresponding Gr\"ossencharacter.

\begin{proposition}
 For any $N < \infty$, the set of points $\chi \in \mathcal{E}(U)$ with $c(\chi) < N$ (or with values in a given coefficient field $E$) is contained in a finite union of 1-dimensional closed subschemes of $\mathcal{E}(U)$. The set of all locally algebraic points is not contained in any proper closed subscheme of $\mathcal{E}(U)$, but is contained in a 1-dimensional closed subvariety of the generic fibre $\mathcal{E}(U)^\rig$.
\end{proposition}

It follows that a rigid-analytic function on $\mathcal{E}(U)^\rig$ is not determined by its values at locally algebraic characters, but that a \emph{bounded} rigid-analytic function is determined by these values. 

\section{Sketch of the \texorpdfstring{$\GL_2$}{GL(2)} theory}

We now suppose $K$ is an imaginary quadratic field in which $p$ splits, and $\mathfrak{N}$ an integral ideal of $\cO_K$ prime to $p$. For integers $a, b \ge 2$, we let $S_{a, b}(\Gamma_1(\mathfrak{N}p^r))$ denote the space of cuspidal cohomological automorphic forms for $\GL_2 / K$ of weight $(a, b)$ and level $\Gamma_1(\mathfrak{N}p^r)$; equivalently, this is the space $H^1_\mathrm{par}(Y_1(\mathfrak{N}p^r), \mathcal{V}_{a, b})$ where $Y_1(\mathfrak{N}p^r)$ is the appropriate arithmetic quotient of $\GL_2(\mathbb{A}_K)$ and $\mathcal{V}_{a, b}$ is the locally constant sheaf on $Y_1(\mathfrak{N}p^r)$ corresponding to the algebraic representation of $\operatorname{Res}_{K / \mathbb{Q}} \GL_2$ of highest weight $(a, b)$. This is a finite-dimensional vector space over $K$.

We fix a choice of prime $\mathfrak{p} \mid p$, and hence an embedding $K \into \QQ_p$. For a locally constant character $\chi$ of $\cO_{K, p}^\times$ of conductor $c$, with values in a finite extension $E$ of $\QQ_p$, we let $S_{a, b}(\Gamma_1(\mathfrak{N}p^c), \chi)_\mathrm{ord}$ denote the subspace of $S_{a, b}(\Gamma_1(\mathfrak{N}p^c)) \otimes_{K} E$ of forms on which the diamond operators act via $\chi$ and which are ordinary at $\mathfrak{p}$ and $\overline{\mathfrak{p}}$.

We say that a locally algebraic weight $\kappa : x \mapsto x^a \overline{x}^b \varepsilon(x) \in \mathcal{W}$, with $\varepsilon$ of finite order, is \emph{arithmetical} if $a, b \ge 2$. 

\begin{theorem}[{\cite[Theorem 3.2]{hida94}}]
 There exists a finitely-generated $\ZZ_p\llbracket\cO_{K, p}^\times\rrbracket$-module $\mathbb{H}$ such that for any arithmetical character $\kappa$ as above,
 \[ S_{a, b}(\Gamma_1(\mathfrak{N}p^c, \varepsilon)_\mathrm{ord} \cong \mathbb{H} \otimes_{\ZZ_p\llbracket\cO_{K, p}^\times\rrbracket, \kappa} E.\]
\end{theorem}

The sheaf $\mathbb{H}$ on $\mathcal{W}$ corresponds to the pushforward of $\cO(\mathcal{E}(U))$ to $\mathcal{W}$ in the $\GL_1$ theory. Hida has given a characterisation of its geometry analogous to proposition \ref{dimension} above:

\begin{theorem}[{\cite[Theorem 6.2]{hida94}}]
 The support of $\mathbb{H}$ is an equidimensional subscheme of $\mathcal{W}$ of dimension 1.
\end{theorem}

We also have an obstruction to the existence of locally algebraic points arising from archimedean considerations, analogous to Theorem \ref{weilthm}:

\begin{theorem}[{\cite[3.6.1]{harder87}}]
 The space $S_{a, b}(\Gamma_1(\mathfrak{N}p^c))$ is zero unless $a = b$.
\end{theorem}

Hence any arithmetical weight lying in $\operatorname{Supp} \mathbb{H}$ is locally parallel, and thus contained in a translate by some locally constant character of the 1-dimensional subscheme $\mathcal{W}_0$ parametrising parallel weights.

\begin{theorem}[{\cite[Lemma 8.8]{calegari-mazur}}]
There exist pairs $(K, \mathfrak{N}, \mathfrak{p})$ such that $\operatorname{Supp} \mathbb{H}$ has nonempty intersection with, but does not contain, the component of $\mathcal{W}_0$ containing the character $x \mapsto (\operatorname{N}_{K/\QQ} x)^2$.
\end{theorem}

We deduce that in such cases, $\operatorname{Supp} \mathbb{H}$ has irreducible components $X$ of dimension 1 such that for any $N < \infty$, the set of arithmetical weights $\kappa \in X$ with $c(\kappa) < N$ is finite; in particular, there are finitely many arithmetical weights in $X(E)$ for any given field $E$. Furthermore, the set of all arithmetical weights is not dense in the rigid Zariski topology of $X$. However, since the set of locally parallel weights is not formally Zariski-closed in $\mathcal{W}$, one cannot rule out the possibility that the set of all arithmetical locally algebraic weights in $X$ is infinite (and hence dense in $X$ for the formal Zariski topology).

\begin{remark}
 It is asserted in \cite[Theorem 8.9]{calegari-mazur} that there are components $X$ which contain only finitely many arithmetical weights, but the proof given therein relies on the assertion that the intersection of $X$ with the set of locally parallel weights is formally Zariski-dense in $X$ (and hence must be either finite or all of $X$); this is false as the set of all locally parallel weights is not formally Zariski-closed in $\mathcal{W}$, and the rigid space $\mathcal{W}^\rig$ is not quasicompact. Similarly, the arguments of Theorem 7.1 of \textit{op.cit.}~do not show that the Galois-theoretic deformation space constructed therein has finitely many specialisations which are Hodge-Tate with parallel Hodge-Tate weights, but rather the weaker statement that for any $N$ it has finitely many specialisations $V$ which have parallel Hodge-Tate weights and for which the Weil-Deligne representation $D_\mathrm{pst}(V)$ has Artin conductor $< N$ (in particular, it has finitely many crystalline specialisations of parallel weight).
\end{remark}

\begin{remark}
 The essential difference between the $\GL_1$ and $\GL_2$ cases is that in the latter we lack an explicit description of the subscheme $\operatorname{Supp} \mathbb{H}$. Thus in the former case we can show that every component of $\mathcal{E}(U)$ contains infinitely many points corresponding to classical automorphic forms, while in the latter case we cannot rule out the existence of irreducible components of $\operatorname{Supp} \mathbb{H}$ containing only finitely many such points -- we merely assert that the existence of such components has not been proven.
\end{remark}

\section{Acknowledgements}

The author would like to thank Kevin Buzzard for much useful advice, and specifically for alerting him to the results of \cite{weil56}; and Frank Calegari for helpful discussions regarding \cite{calegari-mazur}.

\renewcommand{\MR}[1]{%
  \href{http://www.ams.org/mathscinet-getitem?mr=#1}{MR #1}
}

\end{document}